\documentclass[11pt, a4paper]{article} 

\usepackage{soul} 
\usepackage{color, xcolor} 
\usepackage[margin=1in]{geometry} 
\usepackage{amsfonts, amscd, amssymb, mathtools, mathrsfs, dsfont, bbm, bbding} 
\usepackage[amsmath, amsthm, thmmarks]{ntheorem} 
\usepackage{graphicx, xypic, color, float} 
\usepackage{indentfirst}
\usepackage{lmodern}
\usepackage[T1]{fontenc} 
\usepackage{enumerate, listings, verbatim, paralist}
\usepackage{setspace, xspace}
\usepackage[colorlinks=true, citecolor=blue]{hyperref}
\usepackage{extarrows}
\usepackage{pdfpages}
\usepackage{ulem}

\usepackage{setspace}
\AtBeginDocument{%
\addtolength\abovedisplayskip{{-}0.15\baselineskip}%
\addtolength\belowdisplayskip{{-}0.15\baselineskip}%
\addtolength\abovedisplayshortskip{{-}0.15\baselineskip}%
\addtolength\belowdisplayshortskip{{-}0.15\baselineskip}%
}

\newtheorem{thm}{Theorem} [section]

\newtheorem{lemma}[thm]{Lemma}
\newtheorem{defn}{Definition}[section]
\newtheorem{assmp}{Assumption}[section]
\newtheorem{remark}{Remark}[section]

\numberwithin{equation}{section} 

\renewcommand{\geq}{\geqslant}
\renewcommand{\leq}{\leqslant}

\newcommand{\dt}{\ensuremath{\operatorname{d}\! t}}
\newcommand{\ds}{\ensuremath{\operatorname{d}\! s}}

\newcommand{\dw}{\ensuremath{\operatorname{d}\! W}}

\begin{document}
\title{Competitive optimal portfolio selection under mean-variance criterion} 
\date{\today}
\author{Guojiang Shao\thanks{School of Mathematical Sciences, Fudan University, Shanghai 200433, China (\url{gjshao23@m.fudan.edu.cn}).}
	\and Zuo Quan Xu\thanks{Contact author. Department of Applied Mathematics, The Hong Kong Polytechnic University, Kowloon, Hong Kong SAR, China (\url{maxu@polyu.edu.hk}).}
	\and Qi Zhang\thanks{School of Mathematical Sciences, Fudan University, Shanghai 200433, China and Laboratory of Mathematics for Nonlinear Science, Fudan University, Shanghai 200433, China (\url{qzh@fudan.edu.cn}).}
}

\maketitle
\begin{abstract} 
We investigate a portfolio selection problem involving multi competitive agents, each exhibiting mean-variance preferences. Unlike classical models, each agent's utility is determined by their relative wealth compared to the average wealth of all agents, introducing a competitive dynamic into the optimization framework. To address this game-theoretic problem, we first reformulate the mean-variance criterion as a constrained, non-homogeneous stochastic linear-quadratic control problem and derive the corresponding optimal feedback strategies. The existence of Nash equilibria is shown to depend on the well-posedness of a complex, coupled system of equations. Employing decoupling techniques, we reduce the well-posedness analysis to the solvability of a novel class of multi-dimensional linear backward stochastic differential equations (BSDEs). We solve a new type of nonlinear BSDEs (including the above linear one as a special case) using fixed-point theory. Depending on the interplay between market and competition parameters, three distinct scenarios arise: (i) the existence of a unique Nash equilibrium, (ii) the absence of any Nash equilibrium, and (iii) the existence of infinitely many Nash equilibria. These scenarios are rigorously characterized and discussed in detail.
\end{abstract}

 \section{Introduction}

The seminal work of Markowitz \cite{markowitz1952modern,markowitz2008portfolio} introduced mean-variance (MV) analysis, establishing a foundational framework for optimizing asset allocation by balancing risk and return. Since then, MV theory has been extended in numerous directions, including the incorporation of stochastic factors, which are central to the models considered in this paper. In parallel, the study of multi-agent games and, more broadly, mean-field games has emerged as a vibrant area in mathematical finance, particularly in the context of multi-agent optimal investment. These models capture the interactions among multiple investors, where each agent seeks to optimize not only their individual wealth but also their relative performance compared to others. Consequently, the agents' decisions are influenced by both their personal risk preferences and the competitive dynamics of the market.

This paper investigates competitive optimal portfolio selection under the MV criterion, bridging game theory and portfolio theory in non-Markovian market setting. In contrast to traditional frameworks where agents optimize in isolation, we consider a scenario in which agents compete to maximize their terminal wealth relative to the average wealth of all agents. This leads naturally to a non-cooperative stochastic differential game formulation.

The continuous-time MV portfolio selection problem has been extensively studied under various market assumptions and methodological approaches. We briefly review some key developments in this area. Li and Zhou \cite{zhou2000continuous} addressed the continuous-time MV portfolio selection problem using the embedding technique and stochastic linear-quadratic (LQ) control theory. Under the constraint of no short-selling, Li, Zhou, and Lim \cite{li2002dynamic} analyzed the MV portfolio selection problem in continuous time via the Hamilton-Jacobi-Bellman (HJB) equation and two Riccati equations. In a complete market with random coefficients, Lim and Zhou \cite{lim2002mean} investigated the continuous-time MV problem using stochastic LQ control and backward stochastic differential equation (BSDE) theory. Subsequent works by Zhou and Yin \cite{zhou2003markowitz} and Xiong and Zhou \cite{xiong2007mean} extended the MV framework to settings with regime switching and partial information, respectively. More recently, Hu, Shi, and Xu \cite{MR4729330} generalized the problem to non-homogeneous stochastic LQ control with random coefficients and regime-switching dynamics, applying their results to asset-liability management under MV criterion.

Meanwhile, optimal investment and reinsurance strategies under relative performance criteria in mean-field and multi-agent games have garnered increasing attention in recent years. The study of portfolio games with relative performance considerations can be traced back to Espinosa and Touzi \cite{espinosa2015optimal}, who examined multi-agent games with portfolio constraints under CARA utility by analyzing the associated systems of quadratic BSDEs. Subsequently, Lacker and Zariphopoulou \cite{lacker2019mean} derived explicit constant equilibrium strategies for both CARA and CRRA utilities in Markovian markets, utilizing HJB equations. Building on these results, extensions to proportional reinsurance and investment were proposed by Bo, Wang, and Zhou \cite{bo2024mean}, as well as He, He, Chen, and Liu \cite{he2023mean}. More recently, Wang, Xu, and Zhang \cite{wang2024competitive} advanced this line of research by investigating competitive portfolio selection in non-Markovian markets, employing quadratic BSDEs to characterize Nash equilibria under both CARA and CRRA utility frameworks.

However, existing results on competitive optimal portfolio selection under the MV criterion are relatively scarce. In fact, research in this area has primarily focused on time-consistent Nash equilibrium strategies, as seen in works such as Guan and Hu \cite{guan2022time}, Zhu, Guan, and Li \cite{zhu2020time}, and Yang, Chen, and Xu \cite{yang2020time}. These studies investigated time-consistent investment and proportional reinsurance strategies for competitive insurers under the MV criterion, utilizing the extended HJB equations developed by Bj{\''o}rk, Khapko, and Murgoci \cite{bjork2017time}. For further related results on portfolio selection based on relative performance, see Deng, Su, and Zhou \cite{deng2024relative}, Fu \cite{fu2023mean}, Fu and Zhou \cite{fuguan2023mean}, Lacker and Soret \cite{lacker2020many}, Liang and Zhang \cite{liang2024time}, and Zhang and Huang \cite{zhang2024optimal}, among others.

Distinct from the aforementioned results on time-consistent Nash equilibrium strategies, the problem addressed in this paper involves time-inconsistent Nash equilibrium strategies under the MV criterion for a multi-agent game, where the extended HJB equations are not applicable. To tackle this game-theoretic problem, we first reformulate it as a constrained stochastic LQ control problem with a non-homogeneous state equation. Employing Lagrange duality, we derive the optimal feedback strategy for each agent. The existence of a Nash equilibrium requires analyzing a coupled system comprising linear optimal feedback controls, forward stochastic differential equations (SDEs), and BSDEs. By applying decoupling techniques, we separate the SDEs from the coupled system and characterize the Nash equilibrium via a novel class of linear multi-dimensional BSDEs with random coefficients:
\begin{equation*} \label{intro}
	\left\{
	\begin{array}{l}
		\mathrm{d}\boldsymbol{h}(t) = -\left\{A(t)\boldsymbol{h}(t) + B(t)\boldsymbol{\eta}(t) + C(t)\boldsymbol{h}(0) + F(t) \right\} \dt + \boldsymbol{\eta}(t) \dw(t), \quad t \in [0, T], \\
		\boldsymbol{h}(T) = 0.
	\end{array}
	\right.
\end{equation*}
Notably, the driver of this BSDE depends on the solution $\boldsymbol{h}(0)$, making it a nonstandard BSDE. To ensure the admissibility of the Nash equilibrium, we establish the solvability of a class of general nonlinear BSDEs (including the above linear one as a special case) using fixed-point theory. Depending on the market and competition parameters, three scenarios may arise: the existence of a unique Nash equilibrium, the absence of any Nash equilibrium, or the existence of infinitely many Nash equilibria. These scenarios are thoroughly analyzed and discussed. It is worth emphasizing that, in contrast to the results in \cite{lacker2019mean}, \cite{guan2022time}, and \cite{wang2024competitive}, our derived strategy depends on both the initial and current values of wealth.

The remainder of this paper is organized as follows. In Section~\ref{sec:pf}, we formulate the competitive optimal portfolio selection problem under MV criterion and turn it into a constrained stochastic LQ control problem. Section~\ref{sxz7} presents the derivation of the optimal strategy for each agent under the MV criterion. In Section~\ref{sxz12}, we address the Nash equilibrium by analyzing a coupled system comprising linear optimal feedback controls, SDEs, and BSDEs. Section~\ref{sec:example} considers a special case to illustrate our theoretical results. Finally, Section~\ref{sec:conclusion} summarizes our findings in a comprehensive table. Appendix~\ref{appendix} provides a proof of the well-posedness for a new class of general nonlinear BSDEs, whose linear form arises in our study.

\section{Problem Formulation}\label{sec:pf}

Let $(\Omega, \mathcal{F}, \mathbb{P})$ be a complete probability space, and $\{W(t)\}_{t\in[0,T]}$ is a one-dimensional Brownian motion on it. 
Denote by $\mathbb{F}=\{\mathcal{F}_t\}_{t\in[0,T]}$ the filtration generated by $W$.
For $n\in\mathbb{N}$, we define some useful spaces as follows. 
\begin{align*}
\begin{array}{rl} 
L_{\mathbb{F}}^{\infty}\left(0, T ; \mathbb{R}^n \right): & \text {the set of $\mathbb{F}$-adapted essentially bounded } \mathbb{R}^n \text{-valued processes;}\\ 
L_{\mathbb{F}}^{\infty}\left(0, T ; \mathbb{R}^{n\times n} \right): &  \text {the set of $\mathbb{F}$-adapted essentially bounded } \mathbb{R}^{n\times n} \text{-valued processes};\\
L_{\mathbb{F}}^{\infty}\left(0, T ; \mathbb{R}_{+}\right): &  \text {the set of $\mathbb{F}$-adapted essentially bounded nonnegative processes};\\
L_{\mathbb{F}}^{\infty}\left(0, T ; \mathbb{R}_{>0}\right): &  \text {the set of $\mathbb{F}$-adapted essentially bounded positive processes};\\
L_{\mathbb{F}}^{\infty}\left(0, T ; \mathbb{R}_{\gg 1}\right): &  \text {the set of $\mathbb{F}$-adapted processes } v:[0, T] \times \Omega \rightarrow(0,+\infty) ~\text{such}\\
& \text {that $c^{-1} \leqslant v(t) \leqslant c$ a.e. a.s. for some constant $c>0$};\\
L_{\mathbb{F}}^{\infty}\left(0, T ; \mathbb{R}_{\ll -1}\right): &  \text {the set of $\mathbb{F}$-adapted processes }$ {\small$v:[0, T] \times \Omega \to(-\infty,0)$} $\text{such}\\
& \text {that $c^{-1} \leqslant v(t) \leqslant c$ a.e. a.s. for some constant $c<0$};\\
L_{\mathbb{F}}^2\left(0, T ; \mathbb{R}^n\right): &  \text {the set of $\mathbb{F}$-adapted processes } v:[0, T] \times \Omega \rightarrow \mathbb{R}^n ~\text {such that } \\
& \mathbb{E}\left[\int_0^T|v(t)|^2 \dt\right]<\infty;\\
S_{\mathbb{F}}^2\left(0,T; \mathbb{R}^n\right): &  \text {the set of $\mathbb{F}$-adapted processes $v:[0, T] \times \Omega \rightarrow \mathbb{R}^n$ with} \\
& \text {continuous sample paths such that } \mathbb{E}\left[\sup_{t \in[0, T]}|v(t)|^2\right]<\infty.%
\end{array}
\end{align*}

BMO martingale, which is a short form of the martingale of bounded mean oscillation, plays a big role in this paper. For any $f \in L_{\mathbb{F}}^2(0, T ; \mathbb{R}^1)$, $\int_0^{\cdot} f(s) \dw(s)$ is a BMO martingale on $[0, T]$ if and only if there exists a constant $c>0$ such that
\begin{equation*}
	\mathbb{E}\left[\int_\tau^T|f(s)|^2 \mathrm{~d} s\; \bigg|\; \mathcal{F}_\tau\right] \leqslant c,
\end{equation*}
holds for all $\mathbb{F}$-stopping times $\tau \leqslant T$. We denote the space of BMO martingales by
{\small\begin{equation*}
	L_{\mathbb{F}}^{2, \mathrm{BMO}}(0, T ; \mathbb{R}^n)=\left\{f \in L_{\mathbb{F}}^2(0, T ; \mathbb{R}^n) : \int_0^{\cdot} f(s) \dw(s) ~\text {is a BMO martingale on }[0, T]\right\}.
\end{equation*}}

We now introduce our financial market, in which there is a risk-free asset (bond) and $n\geq 2$ risky assets (stocks). Correspondingly, there are $n$ agents in the market, each of which has its preference for a stock to invest. Consequently,
the dynamic equations of bond $S_0 = \{S_0(t)\}_{t\in[0,T]}$ and the stock $i$ for the agent $i$ $S_i=\{S_i(t)\}_{t\in[0,T]}$ are given by
\begin{equation*}
	\left\{\begin{array}{l}
		\frac{\mathrm{d} S_0(t)}{S_0(t)} = r(t)\dt, \quad S_0(0)=s_0>0, \quad t \in[0, T],\\
		\frac{\mathrm{d} S_i(t)}{S_i(t)}=\mu_i(t)\dt+\sigma_i(t)\dw(t), \quad S_i(0)=s_i>0, \quad t \in[0, T],
	\end{array}\right.
\end{equation*}
where $r\in L_{\mathbb{F}}^{\infty}(0, T;\mathbb{R}_+)$, $\mu_i\in L_{\mathbb{F}}^{\infty}(0, T;\mathbb{R}_+)$ and $\sigma_i\in L_{\mathbb{F}}^{\infty}(0, T;\mathbb{R}_{\gg 1})$ serve as the interest rate, the appreciation rate of stock $i$ and the volatility, respectively.
Our model is non-Markovian since these parameters are stochastic. 

Denote by $\rho_i \triangleq \frac{\mu_i - r}{\sigma_i}\in L_{\mathbb{F}}^{\infty}(0, T;\mathbb{R}^1)$ the risk premium of stock $i$. If $\rho_i\equiv0$, there is no motivation for agent $i$ to invest in stock $i$, hence in the rest of this paper, we always assume $\rho_i\not\equiv0$. For simplicity, we only consider the case that the common noise $W$ in the market is $1$-dimensional Brownian motion.

Denote by $\{\pi_i(t)\}_{t\in[0,T]}$ the amount of money invested in stock $i$. Then the self-financing wealth of agent $i$, $\{X_i(t)\}_{t\in[0,T]}$, is given by
\begin{equation*}
	\left\{\begin{aligned}
		\mathrm{d}X_i(t) 
		&= \left[r(t)X_i(t) + \pi_i(t)\sigma_i(t) \rho_i(t)\right]\dt+\pi_i(t)\sigma_i(t)\dw(t), \quad t \in[0, T],
		\\ X_i(0)&=x_i .
	\end{aligned}\right.
\end{equation*}
\begin{defn}
	A vector portfolio strategy $\boldsymbol{\pi}\triangleq (\pi_1, \pi_2, \cdots, \pi_n)^\top$ is called admissible if $\boldsymbol{\pi} \in L_{\mathbb{F}}^2(0, T; \mathbb{R}^n)$.
\end{defn}
	Set $\mathcal{U} \triangleq L_{\mathbb{F}}^2(0, T; \mathbb{R}^1)$, 
	 then $\boldsymbol{\pi}$ is admissible if and only if $\pi_i\in \mathcal{U}$ for all $i=1,2,\cdots,n$.

In our game, each agent aims to outperform the others.
We assume that every agent uses an MV preference on the relative wealth.
The arithmetic average wealth at time $T$ is defined as
$$
\bar{X}(T) \triangleq \frac{1}{n} \sum_{i = 1}^n X_i(T).
$$
For agent $i$, the relative wealth compared to others is defined as $X_i(T) - \theta_i \bar{X}(T)$, where $\theta_i \in [0, 1]$ is a parameter describing agent $i$'s relative preference between their own wealth and average wealth. The agent $i$'s MV preference is formulated as
\begin{equation} \label{func0}
	J_i(\pi_1, \pi_2, \ldots, \pi_n, \theta_i, \gamma_i) \triangleq \mathbb{E}[X_i(T) - \theta_i \bar{X}(T)] - \frac{\gamma_i}{2} \text{Var}[X_i(T) - \theta_i \bar{X}(T)],
\end{equation}
where $\gamma_i > 0$ is the risk aversion parameter of agent $i$. 

To simplify our problem, we put forward a new cost functional
\begin{equation} \label{func1}
	\hat{J}_i(\pi_1, \pi_2, \ldots, \pi_n, \theta_i, \gamma_i) \triangleq \mathbb{E}[X_i(T) - \theta_i \hat{X}_i(T)] - \frac{\gamma_i}{2} \text{Var}[X_i(T) - \theta_i \hat{X}_i(T)],
\end{equation}
where
\begin{equation*}
	\hat{X}_i(t) \triangleq \frac{\sum_{k\neq i} X_k(t)}{n - 1},
\end{equation*}
satisfying the state equation
\begin{equation*}\label{sxz1}
	\left\{\begin{aligned}
		\mathrm{d}\hat{X}_i(t) &= [r(t)\hat{X}_i(t) + \widehat{(\rho \sigma \pi)}_i(t)]\dt + \widehat{(\sigma\pi)}_i(t) \dw(t), \quad t \in[0, T], \\
		\hat{X}_i(0) &= \hat{x}_i,
	\end{aligned}\right.
\end{equation*}
with 
\begin{align*}
	\widehat{(\rho \sigma \pi)}_i(t) &\triangleq \frac{1}{n - 1}\sum_{k\neq i} \rho_k(t) \sigma_k(t) \pi_k(t),\\
\widehat{(\sigma\pi)}_i(t) &\triangleq \frac{1}{n - 1}\sum_{k\neq i} \sigma_k(t)\pi_k(t),\\
\hat{x}_i &\triangleq \frac{1}{n - 1}\sum_{k\neq i} x_k.
\end{align*}
	
A direct computation reveals the relation between the two cost functionals (\ref{func0}) and (\ref{func1}):
	\begin{equation*}
		J_i\left(\pi_1, \pi_2, \ldots, \pi_n, \theta_i, \gamma_i\right) = \left(1 - \frac{\theta_i}{n}\right) \hat{J}_i\left(\pi_1, \pi_2, \ldots, \pi_n, \frac{(n-1)\theta_i }{n - \theta_i}, \left(1 - \frac{\theta_i}{n}\right)\gamma_i\right),
	\end{equation*}
	where \(\left(1 - \frac{\theta_i}{n}\right)\gamma_i>0\) and $\frac{(n-1)\theta_i }{n - \theta_i}$ monotonically increases from \(0\) to \(1\) as \(\theta_i\) increases from \(0\) to \(1\). Hence optimizing the cost functional \eqref{func0} is equivalent to optimizing the cost functional \eqref{func1} with a trivial difference in parameters $\theta_i$ and $\gamma_i$. For simplicity, we focus on the cost functional \eqref{func1} in the rest of the paper. We define Nash equilibrium as follows. 
\begin{defn} \label{nashdefinition}
	An admissible vector strategy \( \boldsymbol{\pi}^* = (\pi_1^*,\pi_2^*, \ldots, \pi_n^*)^\top \) is called a Nash equilibrium (strategy) if, for each agent \( i \in \{1, \ldots, n\} \) and any \( \pi_i \in \mathcal{U} \),
	\begin{equation*}
		\hat{J}_i\left(\pi_1^*, \ldots, \pi_{i-1}^*, \pi_i^*, \pi_{i-1}^*, \ldots, \pi_n^*; \theta_i, \gamma_i\right) \geq \hat{J}_i\left(\pi_1^*, \ldots, \pi_{i-1}^*, \pi_i, \pi_{i+1}^*, \ldots, \pi_n^*; \theta_i, \gamma_i\right).
	\end{equation*}
\end{defn}

In the rest of this section, we further simplify our model. 
Set \(Z_i(t) \triangleq X_i(t) - \theta_i \hat{X}_i(t)\) as a new state variable. Then it satisfies the dynamic equation
\begin{equation} \label{state1}
	\left\{
	\begin{aligned}
		\mathrm{d}Z_i(t)= &\left[ r(t)Z_i(t) + \rho_i(t)\sigma_i(t)\pi_i(t) - \theta_i \widehat{(\rho \sigma \pi)}_i(t) \right] \dt\\
&+ \left[ \sigma_i(t)\pi_i(t) - \theta_i \widehat{(\sigma \pi)}_i(t) \right] \dw(t), \quad t \in[0, T],\\
		Z_i(0)= & z_i \triangleq x_i - \theta_i \hat{x}_i.
	\end{aligned}
	\right.
\end{equation}

When the \( n-1 \) agents' strategies \( \pi_k \in \mathcal{U} \), \( k \neq i \), are fixed, the game problem for agent $i$ reduces to an MV portfolio selection problem:
\begin{equation} \label{prob1}
	\begin{array}{rl}
		\max \limits_{\pi_i} & \mathbb{E}[Z_i(T)] - \frac{\gamma_i}{2} \operatorname{Var}(Z_i(T)), \\
		\text{subject to} & \left\{
		\begin{array}{l}
			\pi_i \in \mathcal{U}, \\
			(Z_i, \pi_i) \text{ satisfies } \eqref{state1}.
		\end{array}
		\right.
	\end{array}
\end{equation}
As its cost functional involves $\operatorname{Var}(\cdot)$, it is a mean field stochastic control problem. To avoid using the dedicated mean field stochastic control theory, we introduce the following constrained stochastic control problem, parameterized by a target \( d \in \mathbb{R}^1 \):
\begin{equation} \label{prob2}
	\begin{array}{rl}
		\min \limits_{\pi_i} & \operatorname{Var}(Z_i(T)) = \mathbb{E}\left[Z_i(T) - d\right]^2, \\
		\text{subject to} & \left\{
		\begin{array}{l}
			\mathbb{E}[Z_i(T)] = d, \\
			\pi_i \in \mathcal{U}, \\
			(Z_i, \pi_i) \text{ satisfies } \eqref{state1}.
		\end{array}
		\right.
	\end{array}
\end{equation}

Since \eqref{prob2} is a convex optimization problem, we can introduce a Lagrange multiplier \( \lambda \in \mathbb{R}^1 \) to deal with the goal constraint \( \mathbb{E}[Z_i(T)] = d \). Then \eqref{prob2} can be further transformed into an unconstrained stochastic control problem, i.e. for each fixed \( \lambda \),
\begin{equation}\label{sxz5}
\begin{array}{rl}
	\min \limits_{\pi_i} & \mathbb{E}\left[|Z_i(T) - d|^2\right] + 2\lambda\left(\mathbb{E}[Z_i(T)] - d\right)\triangleq J_i(\pi_i,\lambda), \\
	\text{subject to} & \left\{
	\begin{array}{l}
		\pi_i \in \mathcal{U}, \\
		(Z_i, \pi_i) \text{ satisfies } \eqref{state1},
	\end{array}
	\right.
\end{array}
\end{equation}
where the constant \( 2 \) in front of \( \lambda \) is used to complete the square. As a result, above control problem is equivalent to
\begin{equation} \label{prob3}
	\begin{array}{rl}
		\min \limits_{\pi_i} & \mathbb{E}\left[|Z_i(T) - b|^2\right]\triangleq\mathscr{J}_i(\pi_i) ,\\
		\text{subject to} & \left\{
		\begin{array}{l}
			\pi_i \in \mathcal{U}, \\
			b = d - \lambda ,\\
			(Z_i, \pi_i) \text{ satisfies } \eqref{state1}.
		\end{array}
		\right.
	\end{array}
\end{equation}
Therefore, to solve the MV portfolio selection problem \eqref{prob1}, the key is to first solve the stochastic LQ control problem \eqref{prob3}.

\section{Solutions for the MV Problems \eqref{prob1}-\eqref{prob3}}\label{sxz7}

In this section, we fix the strategies \( \pi_k \in \mathcal{U} \), \( k \neq i \), of the \( n-1 \) agents, and solve the MV portfolio selection problems \eqref{prob1}-\eqref{prob3} for agent $i$. 

We first introduce two useful BSDEs,
\begin{equation} \label{bsde1}
	\left\{\begin{aligned}
		& \mathrm{d}p_i = - \left( \left( 2r - \rho_i^2 \right) p_i-2 \rho_i \Lambda_i- \frac{\Lambda_i^2}{p_i} \right) \dt + \Lambda_i \dw, \quad t \in[0, T], \\
		& p_i(T)= 1,
	\end{aligned}\right.
\end{equation}
and
\begin{equation} \label{bsde2}
	\left\{\begin{array}{l}
		\mathrm{d}h_i=\left\{rh_i + \theta_i \widehat{(\rho \sigma \pi)}_i - \theta_i \rho_i \widehat{(\sigma\pi)}_i + \rho_i \eta_i\right\} \dt+\eta_i \dw, \quad t \in[0, T], \\
		h_i(T)= -(d-\lambda).
	\end{array}\right.
\end{equation}
Here and hereafter, we may omit time variables in equations and formulas if it does not cause confusion.

\begin{lemma}
	BSDE \eqref{bsde1} admits a unique solution $(p_i,\Lambda_i)\in L_{\mathbb{F}}^{\infty}\left(0, T ; \mathbb{R}_{\gg 1}\right) \times L_{\mathbb{F}}^{2, \mathrm{BMO}}(0, T ; \mathbb{R}^1)$. Furthermore, \( p_i(t) \) is explicitly given by
	\begin{equation} \label{pit}
		p_i(t) =\frac{1}{\mathbb{E}\left[ \exp\left( \int_t^T -2\rho_i(s)\dw(s) + \int_t^T \left(-2r(s) - |\rho_i(s)|^2\right)\ds \right) \bigg| \mathcal{F}_t \right] }, \quad t \in[0, T].
	\end{equation}
\end{lemma}
\begin{proof}

The first part of the claim follows from Theorem 3.2 in \cite{wang2024competitive}. To establish \eqref{pit}, observe that \((\check{p}_i,\check{\Lambda}_i) \triangleq \left( \frac{1}{p_i}, -\frac{\Lambda_i}{p_i^2} \right) \) solves the linear BSDE
	\begin{equation*}
		\left\{\begin{aligned}
			& \mathrm{d}\check{p}_i = - \left(- \left( 2r - \rho_i^2 \right) \check{p}_i-2 \rho_i \check{\Lambda}_i \right) \dt + \check{\Lambda}_i \dw, \quad t \in[0, T], \\
			& \check{p}_i(T)= 1.
		\end{aligned}\right.
	\end{equation*}
By a simple change of measure, we obtain the explicit expression:
\begin{equation}\label{sxz15}
	\check{p}_i(t) = \mathbb{E}\left[ \exp\left( \int_t^T -2\rho_i(s)\dw(s) + \int_t^T \left(-2r(s) - |\rho_i(s)|^2\right)\ds \right)\; \Bigg|\; \mathcal{F}_t \right].
\end{equation}
So \( p_i(t) = \frac{1}{\check{p}_i(t)} \) yields \eqref{pit}.
\end{proof}

\begin{lemma}
	BSDE \eqref{bsde2} admits a unique solution $$(h_i,\eta_i)\in S_{\mathbb{F}}^2\left(0,T; \mathbb{R}^1\right) \times L_{\mathbb{F}}^{2}(0, T ; \mathbb{R}^1).$$
\end{lemma}
\begin{proof}
	Notice $r\in L_{\mathbb{F}}^{\infty}\left(0, T ; \mathbb{R}_{+}\right)$, $\rho_i \in L_{\mathbb{F}}^{\infty}(0, T ; \mathbb{R}^1)$ and $\theta_i \widehat{(\rho \sigma \pi)}_i - \theta_i \rho_i \widehat{(\sigma\pi)}_i \in L_{\mathbb{F}}^2\left(0, T ; \mathbb{R}^1\right)$. The conclusion follows immediately.
\end{proof}

For now on, we fix the solutions $(p_i,\Lambda_i)$ for BSDE \eqref{bsde1} and $(h_i,\eta_i)$ for BSDE \eqref{bsde2}. Based on them, we now introduce a non-homogeneous linear SDE: 
\begin{equation} \label{state4}
		\left\{\begin{aligned}
			\mathrm{d}{Z}^*_i = & \left( r{Z}^*_i+ \theta_i \rho_i\widehat{(\sigma\pi)}_i-\theta_i \widehat{(\rho \sigma \pi)}_i -\rho_i [\eta_i + ( \frac{\Lambda_i}{p_i}+\rho_i)({Z}^*_i+h_i)] \right) \dt \\
			& - \left( \eta_i + ( \frac{\Lambda_i}{p_i}+ \rho_i )({Z}^*_i+h_i) \right) \dw, \quad t \in[0, T],\\
			{Z}^*_i(0) &= z_i.
		\end{aligned}\right.
	\end{equation}
Note this SDE has unbounded coefficients, so its solvability is not immediately ready. 
\begin{lemma}\label{sxz2}
SDE \eqref{state4} admits a solution ${Z}^*_i \in L_{\mathbb{F}}^2\left(0, T ; \mathbb{R}^1\right)$.
\end{lemma}
\begin{proof}
	Clearly, the following SDE with bounded coefficients admits a unique strong solution \( Y_i \in S_{\mathbb{F}}^2(0,T; \mathbb{R}^1) \):
	\begin{equation} \label{lsde}
		\left\{\begin{aligned}
			& \mathrm{d}Y_i = -rY_i \dt - \rho_iY_i \dw, \quad t \in[0, T], \\
			& Y_i(0)= p_i(0)\left(z_i+h_i(0)\right).
		\end{aligned}\right.
	\end{equation}
Applying It\^{o}'s formula, one can see 
\begin{equation}\label{sxz9}
{Z}^*_i \triangleq \frac{Y_i}{p_i} - h_i,
\end{equation} 
is a solution in $ L_{\mathbb{F}}^2(0, T; \mathbb{R}^1)$ to the original SDE \eqref{state4}. Since the above linear transformation is invertible, the uniqueness follows. 
\end{proof}

Now we are ready to solve the stochastic LQ control problem \eqref{prob3}.
\begin{thm}\label{sxz6}
	The stochastic LQ control problem \eqref{prob3} is well-posed, with the unique optimal feedback control given by 
	\begin{equation*} 
		\pi_i^*(t, Z_i) = \theta_i \frac{\widehat{(\sigma\pi)}_i }{\sigma_i} - \frac{1}{\sigma_i} \left[ \eta_i + ( \frac{\Lambda_i}{p_i}+\rho_i)(Z_i+h_i) \right],
	\end{equation*}
	and its corresponding optimal cost functional given by
	\begin{equation*}
		\mathscr{J}_i[\pi_i^*]=p_i(0)|z_i+h_i(0)|^2,
	\end{equation*}
	where $\pi_i^*=\pi^*_i(t, Z_i^*)$ and $Z^*_i$ is determined by \eqref{state4}. 
\end{thm}
\begin{proof}
One can check the pair $(\pi^*_i, Z^*_i)$ satisfies the state equation \eqref{state1}.
Applying Lemma \ref{sxz2}, we can get $\pi_i^*\in \mathcal{U}$.
	For any \( \pi_i \in \mathcal{U}\), let $Z_i$ denote the corresponding state determined by \eqref{state1}. 
	Applying It\^{o}'s formula to $p_i|Z_i+h_i|^2$,
	we have
	\begin{equation*} 
\begin{aligned}
		\mathrm{d}\left(p_i(Z_i+h_i)^2\right) =& p_i\sigma_i^2 |\pi_i-\pi_i^*|^2 \dt\\ 
&+ \left(|Z_i+h_i|^2\Lambda_i+2p(Z_i+h_i)(\sigma_i\pi_i-\theta_i \widehat{(\sigma\pi)}_i + \eta_i) \right)\dw.
\end{aligned}
	\end{equation*}
It yields that
	\begin{equation*}
		\begin{aligned}
			\mathscr{J}_i(\pi_i) =p_i(0)|z_i+h_i(0)|^2+ \mathbb{E}\left[ \int_0^T p_i\sigma_i^2 |\pi_i-\pi_i^*|^2 \dt \right].
		\end{aligned}
	\end{equation*}
	Since $(\pi^*_i, Z^*_i)$ satisfies the state equation \eqref{state1}, it implies that 
	\begin{equation*}
		\mathscr{J}_i[\pi_i^*]=p_i(0)|z_i+h_i(0)|^2.
	\end{equation*}
The above two equations confirm the optimality of 	$\pi_i^*$. 
\end{proof}

Next we turn to the constrained optimization problem \eqref{prob2}. We first establish the feasible condition for it, i.e. for a given target \( d \), there exists an admissible portfolio \( \pi_i \in \mathcal{U}\) satisfying \( \mathbb{E}[Z_i(T)] = d \).
\begin{thm}\label{sxz4}
	The constrained LQ Problem \eqref{prob2} is feasible for any $d \in \mathbb{R}^1$ if and only if
	\begin{equation} \label{feasible}
		\mathbb{E} \int_0^T\left|\rho_i(t) \psi(t)+\xi(t) \right|^2 \dt>0,
	\end{equation}
	where $(\psi,\xi) \in L_{\mathbb{F}}^{\infty}\left(0, T ; \mathbb{R}_{\gg 1}\right) \times L_{\mathbb{F}}^{2}(0, T ; \mathbb{R}^1)$ is the unique solution to the linear BSDE
	\begin{equation*}
		\left\{\begin{array}{l}
			\mathrm{d} \psi=-r \psi \dt+\xi\dw, \quad t \in[0, T], \\
			\psi(T)= 1.
		\end{array}\right.
	\end{equation*}
\end{thm}
\begin{proof}
	Since $\sigma_i\in L_{\mathbb{F}}^{\infty}\left(0, T ; \mathbb{R}_{\gg 1}\right)$, we claim that the feasible condition \eqref{feasible} is equivalent to
	\begin{equation} \label{equi}
		\mathbb{E} \int_0^T\left|\rho_i(t)\sigma_i(t) \psi(t)+\sigma_i(t)\xi(t) \right|^2 \dt>0.
	\end{equation}
	For any admissible $\pi_i\in \mathcal{U}$ and $\beta\in\mathbb{R}^1$, define the scaled portfolio $\pi_i^\beta\triangleq\beta \pi_i$. Denote by $Z_i^\beta$ the wealth process corresponding to $\pi_i^\beta$. Then for $t\in[0,T]$, $Z_i^\beta(t)=\beta x(t)+y(t)$, where \( x\) and \( y\) satisfy
	\begin{equation*}
		\left\{\begin{aligned}
			\mathrm{d}x&= \left( rx+ \rho_i \sigma_i\pi_i \right)\dt + \sigma_i\pi_i \dw, \quad t \in[0, T], \\
			\quad x(0)& = 0,
		\end{aligned}\right.
	\end{equation*}
and
	\begin{equation*}
		\left\{\begin{aligned}
			\mathrm{d}y &= \left[ ry -\theta_i \widehat{(\rho \sigma \pi)}_i \right] \dt - \theta_i\widehat{(\sigma\pi)}_i\dw, \quad t \in[0, T], \\
			\quad y(0)& = z_i.
		\end{aligned}\right.
	\end{equation*}
	Then we have $ \mathbb{E}[Z_i^\beta(T)] = \beta \mathbb{E}[x(T)] +\mathbb{E}[y(T)] $, where $ \mathbb{E}[y(T)]$ is independent of $\pi_i$ and
	\begin{equation*}
		\mathbb{E}[x(T)] = \mathbb{E} \int_{0}^{T}[\rho_i(t) \psi(t)+\xi(t)]\sigma_i(t)\pi_i(t)\dt.
	\end{equation*}
	
We first prove the ``if'' part. For $t\in[0,T]$, taking $\pi_i(t) = \sigma_i(t)[\rho_i(t) \psi(t)+\xi(t)]$ in above equality, we have
	\begin{equation*}
		\mathbb{E}[x(T)] = \mathbb{E} \int_0^T\left|\rho_i(t)\sigma_i(t) \psi(t)+\sigma_i(t)\xi(t) \right|^2 \dt>0.
	\end{equation*}
Hence for any $d \in \mathbb{R}^1$, there exists $\beta \in \mathbb{R}^1$ such that $\mathbb{E}[Z_i^\beta(T)] = d$, and thus $\pi_i^\beta\in \mathcal{U}$ satisfying \( \mathbb{E}[Z_i(T)]=\mathbb{E}[x(T)] +\mathbb{E}[y(T)] = d \).
	
For ``only if'' part, assume that problem \eqref{prob2} is feasible for any $d \in \mathbb{R}^1$, then there exists a $\pi_i\in \mathcal{U}$ such that
	\begin{equation*}
		\mathbb{E}[x(T)] = \mathbb{E} \int_{0}^{T}[\rho_i(t) \psi(t)+\xi(t)]\sigma_i(t)\pi_i(t)\dt \neq 0,
	\end{equation*}
which implies that \eqref{equi} is true.
\end{proof}

\begin{remark}
From Theorem \ref{sxz4}, we see that if the feasible condition \eqref{feasible} does not hold, there is only one feasible target $d$ for the constrained LQ Problem (\ref{prob2}). To avoid this trivial case, we always assume the feasible condition \eqref{feasible} holds from now on, which allows us to deal with the constraint \( \mathbb{E}[Z_i(T)] = d \) by the Lagrangian method.
\end{remark}

To move forward on the solvability of constrained optimization problem \eqref{prob2}, we decompose BSDE \eqref{bsde2} into two components as follows: 
\begin{equation*}
	h_i=\tilde{h}_i+(d-\lambda)\check{h}_i \ \ {\rm and}\ \ \eta_i=\tilde{\eta}_i+(d-\lambda)\check{\eta}_i,
\end{equation*}
where \( (\tilde{h}_i, \tilde{\eta}_i) \) and \( (\check{h}_i, \check{\eta}_i) \) solve the following two linear BSDEs
\begin{equation} \label{bsde3}
	\left\{\begin{array}{l}
		\mathrm{d}\tilde{h}_i=\left[r\tilde{h}_i + \theta_i \widehat{(\rho \sigma \pi)}_i - \theta_i \rho_i \widehat{(\sigma\pi)}_i + \rho_i \tilde{\eta}_i \right] \dt+\tilde{\eta}_i \dw, \quad t \in[0, T], \\
		\tilde{h}_i(T)= 0,
	\end{array}\right.
\end{equation}
and
\begin{equation} \label{bsde4}
	\left\{\begin{array}{l}
		\mathrm{d}\check{h}_i=\left(r\check{h}_i + \rho_i \check{\eta}_i\right) \dt+\check{\eta}_i \dw, \quad t \in[0, T], \\
		\check{h}_i(T)= -1.
	\end{array}\right.
\end{equation}

For the stochastic LQ control problem \eqref{prob3}, the unique optimal feedback control can be written as
\begin{equation*}
	\pi_i^* = \theta_i \frac{\widehat{(\sigma\pi)}_i }{\sigma_i} - \frac{1}{\sigma_i}\left[ \tilde{\eta}_i+(d-\lambda)\check{\eta}_i + ( \frac{\Lambda_i}{p_i}+\rho_i)(Z^*_i+\tilde{h}_i+(d-\lambda)\check{h}_i) \right],
\end{equation*}
with the corresponding cost functional
\begin{equation*}
	\mathscr{J}_i[\pi_i^*]=p_i(0)|z_i+h_i(0)|^2 = p_i(0)|z_i+\tilde{h}_i(0)+(d-\lambda)\check{h}_i(0)|^2.
\end{equation*}

By Proposition 3.5 in \cite{lim2002mean}, the inequality \( p_i(0) \check{h}_i(0) ^2 < 1 \) holds.
By Proposition 4.1 in \cite{lim2002mean}, BSDE \eqref{bsde4} admits a unique solution \( (\check{h}_i, \check{\eta}_i)\in L_{\mathbb{F}}^\infty(0, T; \mathbb{R}_{\ll -1}) \times L_{\mathbb{F}}^2(0, T; \mathbb{R}^1) \), and
	\begin{equation} \label{hit}
		\begin{aligned}
			\check{h}_i(t) = -\mathbb{E} \bigg[ \exp\bigg(\int_{t}^{T} -\rho_i \dw+ \int_{t}^{T}(-r- \frac{\rho_i^2}{2} ) \ds \bigg) \;\bigg{|}\; \mathcal{F}_t \bigg].
		\end{aligned}
	\end{equation}

The assumption on the feasible condition \eqref{feasible} allows us to solve the constrained control problem \eqref{prob2} by solving \eqref{prob3} for a fixed mean $\mathbb{E}[Z_i(T)] = d$. By the Lagrange duality theorem, 
the minimization problem \eqref{prob2} is equivalent to the unconstrained maximization problem
\begin{equation*}
	\min _{\pi_i(\cdot) \in \mathcal{U}, \mathbb{E}[Z_i(T)] = d} \operatorname{Var}(Z_i(T))=\max _{\lambda \in \mathbb{R}^1} \min _{\pi_i(\cdot) \in \mathcal{U}} J_i(\pi_i,\lambda).
\end{equation*}
In particular, by \eqref{sxz5}, \eqref{prob3} and Theorem \ref{sxz6}, we have
\begin{equation*}
\min _{\pi_i(\cdot) \in \mathcal{U}} J_i(\pi_i,\lambda)=	J_i(\pi_i^*,\lambda) = \mathscr{J}_i[\pi_i^*(\cdot)] - \lambda^2 = p_i(0)|z_i+\tilde{h}_i(0)+(d-\lambda)\check{h}_i(0)|^2 - \lambda^2.
\end{equation*}
Thanks to \( p_i(0) \check{h}_i(0) ^2 < 1 \), the maximum of \( \lambda\mapsto J_i(\pi_i^*, \lambda) \) is attained at the optimal Lagrange multiplier
\begin{equation*}
\lambda^* = \frac{p_i(0)\check{h}_i(0)\left(z_i + \tilde{h}_i(0) + \check{h}_i(0)d\right)}{p_i(0)|\check{h}_i(0)|^2 - 1},
\end{equation*}
which gives the optimal value of \eqref{prob2} as 
\begin{equation*}
	\frac{p_i(0)|z_i + \tilde{h}_i(0) + \check{h}_i(0)d|^2}{1 - p_i(0) |\check{h}_i(0)|^2}.
\end{equation*}

Finally, let us study the MV portfolio selection problem \eqref{prob1}. We only need to solve
\begin{equation*}
	\max _{d \in \mathbb{R}^1}\left(d - \frac{\gamma_i}{2} \cdot \frac{p_i(0)|z_i + \tilde{h}_i(0) + \check{h}_i(0)d|^2}{1 - p_i(0) |\check{h}_i(0)|^2}\right),
\end{equation*}
which attains its maximum
\begin{equation*}
	\frac{1 - p_i(0)|\check{h}_i(0)|^2}{2 \gamma_i p_i(0)|\check{h}_i(0)|^2} - \frac{z_i + \tilde{h}_i(0)}{\check{h}_i(0)},
\end{equation*}
at the optimal mean
\begin{equation*}
	d^* = \frac{1}{\gamma_i}\left( \frac{1}{p_i(0)|\check{h}_i(0)|^2} - 1 \right) - \frac{z_i + \tilde{h}_i(0)}{\check{h}_i(0)}.
\end{equation*}

In this way, we know
\begin{equation}\label{sxz8}
	d^* - \lambda^* = -\frac{z_i + \tilde{h}_i(0)}{\check{h}_i(0)} + \frac{1}{\gamma_i p_i(0)|\check{h}_i(0)|^2}.
\end{equation}
Therefore, for problem \eqref{prob1}, the linear optimal feedback control for agent $i$ is
\begin{multline} \label{control2}
		\pi_i^* = \theta_i \frac{\widehat{(\sigma\pi^*)}_i }{\sigma_i} - \frac{1}{\sigma_i} \bigg[ \tilde{\eta}_i+(-\frac{z_i+\tilde{h}_i(0)}{\check{h}_i(0)}+\frac{1}{\gamma_ip_i(0)|\check{h}_i(0)|^2})\check{\eta}_i \\
		+( \frac{\Lambda_i}{p_i}+\rho_i)\bigg(Z^*_i+\tilde{h}_i+(-\frac{z_i+\tilde{h}_i(0)}{\check{h}_i(0)}+\frac{1}{\gamma_ip_i(0)|\check{h}_i(0)|^2})\check{h}_i\bigg) \bigg].
\end{multline}

\section{Solving the Nash Equilibrium}\label{sxz12}

In Section \ref{sxz7}, we deal with the mena-variance problem by fixing the other $n-1$ agents' strategies, while in this section, we study the Nash equilibrium of the MV portfolio selection problem \eqref{prob1}. It means that we need to find out a vector portfolio strategy \( \boldsymbol{\pi^*}=(\pi^*_1, \pi^*_2, \cdots, \pi^*_n)^\top\in L_{\mathbb{F}}^2(0, T; \mathbb{R}^n) \) such that \eqref{control2} is satisfied for all $i=1,2,\cdots,n$.

Needless to say, it is much more complicated to solve the Nash equilibrium than to solve a single MV problem. Notice that the coefficients of BSDEs \eqref{bsde1} and BSDEs \eqref{bsde4} only depend on market parameters rather than portfolio strategies, but SDEs \eqref{state1}, BSDEs \eqref{bsde3} and linear optimal feedback controls \eqref{control2} constitute a coupled system. Hence the key point is to establish the well-posedness of this coupled system.

We start from decoupling SDEs \eqref{state1}. 
Due to \eqref{sxz8}, we have
\begin{equation*}
	h_i(0) = \tilde{h}_i(0) + \check{h}_i(0) \left( -\frac{z_i + \tilde{h}_i(0)}{\check{h}_i(0)} + \frac{1}{\gamma_i p_i(0)|\check{h}_i(0)|^2} \right) = -z_i + \frac{1}{\gamma_i p_i(0)\check{h}_i(0)}.
\end{equation*}
Rewritten SDEs \eqref{lsde} as (here and hereafter, we use $Y^*_i$ instead of $Y_i$ in the discussion of Nash equilibrium):
\[
\left\{
\begin{aligned}
	\mathrm{d}Y^*_i &= -rY^*_i\dt - \rho_iY^*_i\dw, \quad t \in [0, T], \\
	Y^*_i(0) &= \frac{1}{\gamma_i \check{h}_i(0)},
\end{aligned}
\right.
\]
with the explicit expression 
\begin{equation*}
	Y^*_i(t) = \frac{1}{\gamma_i \check{h}_i(0)} \exp\left( \int_0^t -\rho_i(s)\dw(s) + \int_0^t \left(-r(s) - \frac{\rho_i(s)^2}{2}\right)\ds \right), \quad t \in [0, T],
\end{equation*}
which only depends on market parameters, and is independent of portfolio strategies of all agents. 
Substituting \eqref{sxz9} into the optimal feedback controls \eqref{control2}, we have
\begin{equation} \label{control3}
	\pi_i^* = \theta_i \frac{\widehat{(\sigma\pi^*)}_i }{\sigma_i} - \frac{1}{\sigma_i} \left[ \tilde{\eta}_i+\left(-\frac{z_i+\tilde{h}_i(0)}{\check{h}_i(0)}+\frac{1}{\gamma_ip_i(0)|\check{h}_i(0)|^2}\right)\check{\eta}_i + ( \frac{\Lambda_i}{p_i}+\rho_i)\frac{Y^*_i}{p_i} \right].
\end{equation}

After decoupling SDEs \eqref{state1} from the coupled system, we obtain the system of linear equations \eqref{control3} of unknown variables \( \pi_i^* \), \( i=1,2,\cdots,n\), coupled with BSDEs \eqref{bsde3}. Next, we further decouple BSDEs \eqref{bsde3} by giving explicit forms of \( \pi_i^* \), $i=1,2,\cdots,n$, with the help of \eqref{control3}. 

For this, define two constants
\begin{equation*}
	\Psi \triangleq \sum_{i=1}^{n} \frac{\theta_i}{n - 1 + \theta_i} \in [0, 1], \quad \hat{\gamma} \triangleq \sum_{i=1}^{n} \frac{1}{\gamma_i},
\end{equation*}
and three average quantities
\begin{equation*}
	\overline{\sigma\pi^*} \triangleq \frac{1}{n}\sum_{k=1}^n \sigma_k\pi^*_k, \quad \bar{x} \triangleq \frac{1}{n}\sum_{k=1}^n x_k, \quad \bar{\rho} \triangleq \frac{1}{n}\sum_{i=1}^{n}\rho_i.
\end{equation*}
Substitute \eqref{control3} into \( \overline{\sigma\pi^*} \), and it yields
\begin{equation}\label{sxz10}
	\sigma_i\pi^*_i = \theta_i \frac{n\overline{\sigma\pi^*} - \sigma_i\pi^*_i}{n - 1} - \phi_i,
\end{equation}
where
\begin{equation*}
	\phi_i \triangleq \tilde{\eta}_i + c_i \tilde{h}_i(0) + f_i,
\end{equation*}
with
\begin{equation*}
	c_i = -\frac{\check{\eta}_i}{\check{h}_i(0)}, \quad f_i = \left(-\frac{z_i}{\check{h}_i(0)} + \frac{1}{\gamma_i p_i(0)|\check{h}_i(0)|^2}\right)\check{\eta}_i + \left( \frac{\Lambda_i}{p_i} + \rho_i \right)\frac{Y^*_i}{p_i}.
\end{equation*}
The equality \eqref{sxz10} implies
\begin{equation}\label{sxz11}
	\sigma_i\pi^*_i = \frac{n\theta_i}{n - 1 + \theta_i} \overline{\sigma\pi^*} - \frac{\phi_i}{1 + \frac{\theta_i}{n - 1}}.
\end{equation}
Sum up all agents, and then divide by \( n \). It turns out that
\begin{equation}\label{average1}
	\overline{\sigma\pi^*} = \Psi \overline{\sigma\pi^*} - \frac{1}{n}\sum_{i=1}^{n} \frac{\phi_i}{1 + \frac{\theta_i}{n - 1}}.
\end{equation}
So it can be seen from above that the solvability of \eqref{control3} depends on the value of $\Psi$. Then we will discuss the solvability of \eqref{control3} in the usual case \(\Psi < 1\) and in the marginal case \(\Psi = 1\), respectively.

\subsection{The usual case}
For \(\Psi < 1\), 
substituting the average control \(\overline{\sigma\pi^*}\) obtained in \eqref{average1} into \eqref{sxz11}, we have
\begin{equation} \label{construction}
	\sigma_i\pi^*_i = -\frac{1}{1-\Psi} \frac{n\theta_i}{n-1+\theta_i}\sum_{i=1}^{n}\frac{\phi_i}{n+\frac{n\theta_i}{n-1}}-\frac{\phi_i}{1+\frac{\theta_i}{n-1}}.
\end{equation}
Then the terms \(\theta_i \widehat{(\rho \sigma \pi^*)}_i - \theta_i \rho_i \widehat{(\sigma \pi^*)}_i\) can be given explicitly as a linear combination of \(\tilde{h}_j\), \(\tilde{\eta}_j\) and \(\tilde{h}_j(0)\) for $j=1,2,\cdots,n$.

Set $\tilde{\boldsymbol{h}} \triangleq (\tilde{h}_1,\tilde{h}_2,\cdots,\tilde{h}_n)^\top$ and $\tilde{\boldsymbol{\eta}} \triangleq (\tilde{\eta}_1,\tilde{\eta}_2,\cdots,\tilde{\eta}_n)^\top$. BSDEs \eqref{bsde3} can be rewritten as
\begin{equation} \label{lbsde1}
	\left\{\begin{array}{l}
		\mathrm{d}\tilde{\boldsymbol{h}}=-\left[A\tilde{\boldsymbol{h}} + B \tilde{\boldsymbol{\eta}} + C\tilde{\boldsymbol{h}}(0) +F \right] \dt+\tilde{\boldsymbol{\eta}} \dw, \quad t \in[0, T], \\
		\tilde{\boldsymbol{h}}(T)= 0,
	\end{array}\right.
\end{equation}
where \(A\), \(B\), and \(C\) are coefficient matrices, 
and \(F\) is a coefficient vector. 
Precisely, for $i,j=1,2,\cdots,n$,
\begin{equation*}
A_{ij}\triangleq\left\{
\begin{aligned}
0,\ \ \ \ \ &i\neq j,\\
-r,\ \ \ &i=j,
\end{aligned}
\right.\ \ \ \ \
B_{ij} \triangleq\left\{
\begin{aligned}
\theta_iM_{ij},\ \ \ \ \ \ \ \ \ &i\neq j,\\
\theta_i M_{ii} - \rho_i,\ \ \ &i=j,
\end{aligned}
\right.
\end{equation*}
\begin{equation*}
C_{ij} \triangleq\left\{
\begin{aligned}
\theta_i M_{ij} c_j,\ \ \ &i\neq j,\\
\theta_i M_{ii} c_i,\ \ \ &i=j,
\end{aligned}
\right.\ \ \ \ \ F_i\triangleq \theta_i (M_{ii}f_i+\sum_{i\neq j}M_{ij}f_j ),
\end{equation*}
with
\begin{equation*}
M_{ij} \triangleq\left\{
\begin{aligned}
\frac{1}{n-1}\frac{1}{1-\Psi}\sum_{k\neq i}\frac{(n-1)\theta_k(\rho_k-\rho_i)}{(n-1+\theta_k)(n-1+\theta_j)} - \frac{(n-1)(\rho_j-\rho_i))}{n-1+\theta_j},\ \ \ &i\neq j,\\
\frac{1}{n-1}\frac{1}{1-\Psi}\sum_{k\neq i}\frac{(n-1)\theta_k(\rho_k-\rho_i)}{(n-1+\theta_k)(n-1+\theta_i)},\ \ \ \ \ \ \ \ \ \ \ \ \ \ \ \ \ \ \ \ \ \ \ \ \ \ \ \ \ &i=j.
\end{aligned}
\right.
\end{equation*}
Obviously $A \in L_{\mathbb{F}}^{\infty}(0, T ; \mathbb{R}^{n\times n})$, $B \in L_{\mathbb{F}}^{\infty}\left(0, T ; \mathbb{R}^{n\times n} \right)$, $C \in L_{\mathbb{F}}^{2}\left(0, T ; \mathbb{R}^{n\times n} \right)$ and $F \in L_{\mathbb{F}}^{2}\left(0, T ; \mathbb{R}^n \right)$.

Note that \eqref{lbsde1} is a new type of BSDE since its driver depends on $\tilde{\boldsymbol{h}}(0)$. In Lemma \ref{newbsde} in Appendix~\ref{appendix}, the solvability of an extended class of general nonlinear BSDEs in the solution space $S_{\mathbb{F}}^2(0, T; \mathbb{R}^n) \times L_{\mathbb{F}}^2(0, T; \mathbb{R}^{n})$ is established by the fixed-point method for sufficiently small \(T > 0\). Given $\tilde{\boldsymbol{h}}(0)$, due to the linear structure of BSDE \eqref{lbsde1}, its explicit solution can be obtained for any given $T>0$.
For this, we introduce an SDE with solution in \( L_{\mathbb{F}}^2(0, T; \mathbb{R}^{n \times n}) \):
\begin{equation*}
\left\{
\begin{aligned}
	\mathrm{d}\Gamma &= \Gamma \left[A \dt + B \dw\right], \quad t \in[0, T],\\
\Gamma(0) &= I_n.
\end{aligned}
\right.
\end{equation*}
For the solution $\Gamma$ to the above SDE, its inverse flow $\Gamma^{-1}$ satisfies another SDE:
\begin{equation*}
\left\{
\begin{aligned}
	\mathrm{d} \Gamma^{-1}&=\Gamma^{-1}[ \left(-A+B^2\right) \dt- B \dw ], \quad t \in[0, T],\\
\Gamma^{-1}(0) &= I_n.
\end{aligned}
\right.
\end{equation*}
Applying It\^{o}'s formula to $\Gamma \tilde{\boldsymbol{h}}$, we have
\begin{equation*}
	\mathrm{d}(\Gamma \tilde{\boldsymbol{h}}) = -\Gamma \left(C \tilde{\boldsymbol{h}}(0) + F \right) \dt + \Gamma(\tilde{\boldsymbol{\eta}}+B \tilde{\boldsymbol{h}}) \dw.
\end{equation*}
Noticing \( \tilde{\boldsymbol{h}}(T) = 0 \), we have
\begin{equation*}
	\tilde{\boldsymbol{h}}(t) = \Gamma^{-1}(t)\mathbb{E}\left[ \int_t^T \Gamma(s) \left(C(s) \tilde{\boldsymbol{h}}(0) + F(s)\right) \ds \,\bigg|\, \mathcal{F}_t \right].
\end{equation*}
In particular,
\begin{equation*}
	\tilde{\boldsymbol{h}}(0) = \mathbb{E}\left[ \int_0^T \Gamma(s) \left( C(s) \tilde{\boldsymbol{h}}(0) + F(s) \right) \ds \right].
\end{equation*}
Set $K \triangleq \mathbb{E}\big[ \int_0^T \Gamma(s) C(s) \ds \big]$ and $ D\triangleq \mathbb{E}\big[ \int_0^T \Gamma(s) F(s) \ds \big] $, then 
\begin{equation*}
	(I_n - K) \tilde{\boldsymbol{h}}(0) = D.
\end{equation*}
The following result studies the well-posedness of BSDE \eqref{lbsde1} in the usual case. 
\begin{thm} \label{equilibrium}
	Assume \(\Psi < 1\). Then the well-posedness of BSDE \eqref{lbsde1} can be classified into the following situations.
	\begin{enumerate}
		\item Unique Solution: If \(I_n - K\) is invertible, there exists a unique consistent initial vector $\tilde{\boldsymbol{h}}(0) = (I_n - K)^{-1} D$. Consequently, BSDE \eqref{lbsde1} admits a unique solution \((\tilde{\boldsymbol{h}}, \tilde{\boldsymbol{\eta}}) \in S_{\mathbb{F}}^2(0, T; \mathbb{R}^n) \times L_{\mathbb{F}}^2(0, T; \mathbb{R}^{n})\).
		\item Infinitely Many Solutions: If \(I_n - K\) is singular and \(D \in \text{Im}(I_n - K)\), there exist infinitely many solutions \((\tilde{\boldsymbol{h}}, \tilde{\boldsymbol{\eta}})\in S_{\mathbb{F}}^2(0, T; \mathbb{R}^n) \times L_{\mathbb{F}}^2(0, T; \mathbb{R}^{n})\). These solutions are characterized by initial vectors $\tilde{\boldsymbol{h}}(0)$ in the affine space
		\begin{equation*}
			\tilde{\boldsymbol{h}}(0) \in \ker(I_n - K) + (I_n - K)^{\dagger}D,
		\end{equation*}
		where \((I_n - K)^{\dagger}\) represents the Moore-Penrose pseudoinverse of \(I_n - K\).
		
		\item No Solution: If \(I_n - K\) is singular but \(D \notin \text{Im}(I_n - K)\), no solution exists to BSDE \eqref{lbsde1}.
		
	\end{enumerate}
\end{thm}

\begin{remark}
	For sufficiently small \(T > 0\), the norm \(\|K\|_{\infty}\) is also sufficiently small due to the integral structure of \(K\). This guarantees the invertibility of \(I_n - K\) and a unique solution to BSDE \eqref{lbsde1}. This is consistent with Lemma \ref{newbsde} which assets that a unique solution to BSDE \eqref{newbsde} (a generalized nonlinear form of BSDE \eqref{lbsde1}) exists for sufficiently small $T>0$.
\end{remark}

Next, we establish the connection between BSDE \eqref{lbsde1} and Nash equilibrium.
\begin{thm} \label{connection}
Assume \(\Psi < 1\). Then there exists a one-to-one correspondence between Nash equilibrium strategies and the solutions to BSDE \eqref{lbsde1}.
\end{thm}
\begin{proof}
Step 1. Nash Equilibrium $\Rightarrow$ Solution to BSDE \eqref{lbsde1}: Assume that there exists a Nash equilibrium strategy \( \boldsymbol{\pi^*}=(\pi^*_1, \pi^*_2, \cdots, \pi^*_n)^\top\in L_{\mathbb{F}}^2(0, T; \mathbb{R}^n) \). By Definition \ref{nashdefinition}, $\boldsymbol{\pi^*}$ satisfies the coupled system composed of SDEs \eqref{state1}, BSDEs \eqref{bsde3} and linear optimal feedback controls \eqref{control2}. Decoupling SDEs \eqref{state1} as what we do at the beginning of Section \ref{sxz12}, we transform the coupled system to a simpler one composed of \eqref{bsde3} and \eqref{control3}. Then, substituting the equilibrium strategy $\boldsymbol{\pi^*}$ into \eqref{control3}, we solve \eqref{control3} to obtain \(\tilde{h}_j\), \(\tilde{\eta}_j\) and \(\tilde{h}_j(0)\) for $j=1,2,\cdots,n$. Consequently, the vector-valued linear BSDE \eqref{lbsde1} with a solution $(\tilde{\boldsymbol{h}},\tilde{\boldsymbol{\eta}})$ follows from BSDEs \eqref{bsde3}, where $\tilde{\boldsymbol{h}}=(\tilde{h}_1,\tilde{h}_2,\cdots,\tilde{h}_n)^\top$ and $\tilde{\boldsymbol{\eta}}=(\tilde{\eta}_1,\tilde{\eta}_2,\cdots,\tilde{\eta}_n)^\top$.
	
Step 2. Solution to BSDE \eqref{lbsde1} $\Rightarrow$ Nash Equilibrium: Assume that BSDE \eqref{lbsde1} admits a solution $ (\tilde{\boldsymbol{h}},\tilde{\boldsymbol{\eta}})$. The strategies \(\pi^*_i\), $i=1,2,\cdots,n$ can be explicitly constructed by \eqref{construction}. The construction of $\boldsymbol{\pi^*}=(\pi^*_1, \pi^*_2, \cdots, \pi^*_n)^\top$ guarantees that it belongs to $L_{\mathbb{F}}^2(0, T; \mathbb{R}^n)$ and satisfies the coupled system composed of SDEs \eqref{state1}, BSDEs \eqref{bsde3} and linear optimal feedback controls \eqref{control2}, and thus $\boldsymbol{\pi^*}$ is a Nash equilibrium.
	
Step 3. Bijectivity: With a known Nash equilibrium, we can construct a solution to BSDE \eqref{lbsde1} by Step 1, and with this solution to BSDE \eqref{lbsde1}, we can retrieve the Nash equilibrium by Step 2. Therefore, the correspondence is invertible.
\end{proof}

We next introduce two new assumptions on the market parameters for further discussions. 
\begin{assmp} \label{rho}
	The Sharpe ratios of all risky assets are identical, but not identical to $0$, i.e.
	\begin{equation*}
		\rho_i(t)=\rho(t)\not\equiv0, \quad \text{for all } i=1,2,\cdots,n, \text { and } \ t\in[0,T].
	\end{equation*}
\end{assmp}
\begin{assmp} \label{deter}
	None of the Sharpe ratios $\rho_i$, $i=1,2,\cdots,n$, is identical to $0$, and the interest rate $r$ are all deterministic processes.
\end{assmp}

\begin{remark}
	Under Assumption \ref{rho}, $ \widehat{(\rho \sigma \pi)}_i - \rho_i \widehat{(\sigma\pi)}_i = 0$ for $i=1,2,\cdots,n$, and BSDEs \eqref{bsde3} admits a unique solution $(\tilde{h}_i, \tilde{\eta}_i)\equiv(0,0)$.
\end{remark}

\begin{remark}
	Under Assumption \ref{deter}, BSDE \eqref{bsde1} admits a unique solutions
$$\left(p_i\left(t\right),\Lambda_i\left(t\right)\right)=\left( \exp\left( \int_t^T \left(2r(s) - |\rho_i(s)|^2\right)\ds \right), 0 \right),
$$
and BSDE \eqref{bsde4} admits a unique solution $\left(\check{h}_i\left(t\right), \check{\eta}_i\left(t\right)\right)=(- \exp\{ \int_{t}^{T}(-r(s) ) \ds \} ,0)$.
\end{remark}

\begin{remark}\label{sxz19}
	Under Assumption \ref{deter}, since $r$ is deterministic, \( \xi\equiv 0 \) and \( \psi > 0 \). Thus \eqref{feasible} also holds due to \( \rho_i \neq 0 \), $i=1,2,\cdots,n$.
\end{remark}

\begin{thm}\label{sxz18}
	Assume Assumption \ref{rho} and $\Psi < 1$. Then there exists a unique Nash equilibrium.
\end{thm}

\begin{proof}
	Under Assumption \ref{rho}, BSDEs \eqref{bsde3} admit only trivial solutions $(0,0)$. BSDEs \eqref{bsde1}, \eqref{bsde2} and \eqref{bsde4} are independent of the agent index \(i\). Hence we always denote their solutions by \(p\), \(\Lambda\), \(h\), \(\eta\), \(\check{h}\), and \(\check{\eta}\) without index $i$ under Assumption \ref{rho}. The optimal strategies \eqref{control2} reduces to
	\begin{equation}\label{sxz13}
\begin{aligned}
		\pi^*_i = \theta_i \frac{\widehat{\sigma\pi^*}}{\sigma_i} - \frac{1}{\sigma_i} \bigg[& \left(-\frac{z_i}{\check{h}(0)} + \frac{1}{\gamma_i p(0)\check{h}(0)^2}\right)\check{\eta}\\
&+ \left(\frac{\Lambda}{p} + \rho\right)\left(Z^*_i + \left(-\frac{z_i}{\check{h}(0)} + \frac{1}{\gamma_i p(0)\check{h}(0)^2}\right)\check{h} \right) \bigg].
\end{aligned}	
\end{equation}
And in this case $\phi_i$ has a form
	\begin{equation} \label{phiz}
		\phi_i = \left(-\frac{z_i}{\check{h}(0)} + \frac{1}{\gamma_i p(0)\check{h}(0)^2}\right)\check{\eta} + \left(\frac{\Lambda}{p} + \rho\right)\left(Z^*_i + \left(-\frac{z_i}{\check{h}(0)} + \frac{1}{\gamma_i p(0)\check{h}(0)^2}\right)\check{h} \right),
	\end{equation}
	or equivalently
	\begin{equation*}
		\phi_i = \left(-\frac{z_i}{\check{h}(0)} + \frac{1}{\gamma_i p(0)\check{h}(0)^2}\right)\check{\eta} + \left(\frac{\Lambda}{p} + \rho\right)\frac{Y^*_i}{p}.
	\end{equation*}
	Substituting \(\phi_i\) into \eqref{sxz13}, we obtain a fully decoupled linear system. Noticing \(\Psi < 1\), we can derive a unique Nash equilibrium from \eqref{construction} based on the fully decoupled linear system.
\end{proof}

\begin{thm} \label{deternashthm}
	Assume Assumption \ref{deter} and $\Psi < 1$. Then there exists a unique Nash equilibrium.
\end{thm}

\begin{proof}
	Under Assumption \ref{deter}, both BSDE \eqref{bsde1} and BSDE \eqref{bsde4} degenerate into ordinary differential equations with deterministic coefficients. Consequently, \(\Lambda_i = 0\), \(\check{\eta}_i = 0\). In the mean time, the coefficient matrices \(A\) and \(B\) are deterministic, and \(C = K = 0\). By Theorems \ref{equilibrium} and \ref{connection}, BSDE \eqref{lbsde1} admits a unique solution which corresponds to the unique Nash equilibrium.
\end{proof}

\subsection{The marginal case}

For \(\Psi = 1\), by definition of $\Psi$, we have \(\theta_i = 1\) for all \(i=1,2,\cdots,n\). Then the equilibrium strategy \eqref{sxz11} reduces to
\begin{equation}\label{sxz14}
	\sigma_i\pi^*_i = \frac{n\theta_i}{n-1+\theta_i}\overline{\sigma\pi^*}-\frac{\phi_i}{1+\frac{\theta_i}{n-1}} = \overline{\sigma\pi^*} - \frac{n-1}{n}\phi_i.
\end{equation}
Set \(\Phi \triangleq \sum_{i=1}^n \phi_i\), then we discuss the existence of Nash equilibrium based on \(\Phi\).
\begin{enumerate}
	\item No Equilibrium: If \(\Phi \neq 0\), it contradicts with \eqref{sxz14} since a sum of \eqref{sxz14} implies $\Phi = 0$.
	\item Uncertain Situation: If \(\Phi = 0\), the existence of Nash equilibrium is uncertain. But if it exists, the equilibrium strategy should be parameterized by a process \(\chi \in L_{\mathbb{F}}^2(0, T; \mathbb{R}^1)\) chosen as a degree of freedom, i.e.
	\[
	\sigma_i\pi^*_i = \chi - \frac{n-1}{n}\phi_i,\ \ \ \ i=1,2,\cdots,n.
	\]
\end{enumerate}

Set $\tilde{\boldsymbol{h}}' = (\tilde{h}'_1,\tilde{h}'_2,\cdots,\tilde{h}'_n)^\top$ and $\tilde{\boldsymbol{\eta}}' = (\tilde{\eta}'_1,\tilde{\eta}'_2,\cdots,\tilde{\eta}'_n)^\top$. BSDEs \eqref{bsde3} can be rewritten as
\begin{equation} \label{lbsde2}
	\left\{\begin{array}{l}
		\mathrm{d}\tilde{\boldsymbol{h}}'=-\left\{A'\tilde{\boldsymbol{h}}' + B' \tilde{\boldsymbol{\eta}}' + C'\tilde{\boldsymbol{h}}'(0) + F' \right\} \dt+\tilde{\boldsymbol{\eta}}' \dw, \quad t \in[0, T], \\
		\tilde{\boldsymbol{h}}'(T)= 0,
	\end{array}\right.
\end{equation}
where \(A'\), \(B'\), and \(C'\) are coefficient matrices, 
and \(F'\) is a coefficient vector. Precisely, for $i,j=1,2,\cdots,n$,
\begin{equation*}
A'_{ij}\triangleq\left\{
\begin{aligned}
0,\ \ \ \ \ &i\neq j,\\
-r,\ \ \ &i=j,
\end{aligned}
\right.\ \ \ \ \
B'_{ij} \triangleq\left\{
\begin{aligned}
\frac{\rho_j-\rho_i}{n},\ \ \ &i\neq j,\\
- \rho_i,\ \ \ \ \ \ \ \ &i=j,
\end{aligned}
\right.
\end{equation*}
\begin{equation*}
C'_{ij} \triangleq\left\{
\begin{aligned}
\frac{\rho_j-\rho_i}{n}c_j,\ \ \ &i\neq j,\\
0,\ \ \ \ \ \ \ \ \ \ \ \ \ &i=j,
\end{aligned}
\right.\ \ \ \ \ F'_i\triangleq \sum_{i\neq j}\frac{\rho_j-\rho_i}{n}f_j - \sum_{i\neq j}\frac{\rho_j-\rho_i}{n-1}\chi.
\end{equation*}
Obviously $A' \in L_{\mathbb{F}}^{\infty}(0, T ; \mathbb{R}^{n\times n})$, $B' \in L_{\mathbb{F}}^{\infty}\left(0, T ; \mathbb{R}^{n\times n} \right)$, $C' \in L_{\mathbb{F}}^{2}\left(0, T ; \mathbb{R}^{n\times n} \right)$ and $F' \in L_{\mathbb{F}}^{2}\left(0, T ; \mathbb{R}^n \right)$.
Set $K' \triangleq \mathbb{E}\big[ \int_0^T \Gamma'(s) C'(s) \ds \big]$ and $D' \triangleq \mathbb{E}\big[ \int_0^T \Gamma'(s) F'(s) \ds \big]$,
where \( \Gamma' \in L_{\mathbb{F}}^2(0, T; \mathbb{R}^{n \times n}) \) solves SDE
\begin{equation*}
\left\{
\begin{aligned}
	\mathrm{d}\Gamma' &= \Gamma' \left[A' \dt + B' \dw\right], \quad t \in[0, T],\\
\Gamma'(0) &= I_n.
\end{aligned}
\right.
\end{equation*}
Then the consistent condition for \(\tilde{\boldsymbol{h}}'(0)\) to guarantee the well-posedness of BSDE \eqref{lbsde2} in the marginal case becomes
\begin{equation*}
	(I_n - K') \tilde{\boldsymbol{h}}'(0) = D'.
\end{equation*}

In general, with a degree of freedom process \(\chi\), it is difficult to guarantee both \(\Phi = 0\) and the well-posedness of BSDE \eqref{lbsde2} in the mean time. Even if the well-posedness of BSDE \eqref{lbsde2} is achieved by the invertibility of \(I_n - K'\) with the help of a sufficiently small time horizon \(T > 0\), to guarantee \(\Phi = 0\) with a suitable choice of \(\chi \in L_{\mathbb{F}}^2(0, T; \mathbb{R})\) is still a challenging problem. If we further know that BSDE \eqref{lbsde2} admits a unique solution, the difficulty is still there due to the required delicate balance between the free-choice parameter \(\chi\) and \(\Phi = 0\) in the marginal case. However, under the homogeneous risk preferences condition (Assumption \ref{rho}) or the deterministic coefficients condition (Assumption \ref{deter}), the coupled system is simplified much, and explicit criteria could be derived to guarantee both \(\Phi = 0\) and the well-posedness of BSDE \eqref{lbsde2}.

We first discuss the delicate balance between \(\chi\) and \(\Phi = 0\) under Assumption \ref{rho}.
\begin{thm} \label{criterion2}
	Assume Assumption \ref{rho} and \(\Psi = 1\). Then the existence of Nash equilibrium can be classified into the following situations.
	\begin{enumerate}
		\item Infinitely Many Equilibria: If the equality
		\begin{equation} \label{equivalent2}
			\mathbb{E}\int_0^T\left( r(s)+\frac{1}{2}|\rho(s)|^2 \right) \ds=\int_0^T\left( r(s)+\frac{1}{2}|\rho(s)|^2 \right) \ds+\int_0^T \rho(s) \dw_s,
		\end{equation}
		holds, then there exist infinitely many Nash equilibria whose components are parameterized by \(\chi \in L_{\mathbb{F}}^2(0, T; \mathbb{R}^1)\) as below
		\[
		\pi^*_i = \frac{\chi}{\sigma_i} - \frac{n-1}{n\sigma_i} \left[ \left(-\frac{z_i}{\check{h}(0)} + \frac{1}{\gamma_i p(0)\check{h}(0)^2}\right)\check{\eta} + \left(\frac{\Lambda}{p} + \rho\right)\frac{Y^*_i}{p} \right].
		\]
		\item No Equilibrium: If \eqref{equivalent2} does not hold, no Nash equilibrium exists.
	\end{enumerate}
\end{thm}
\begin{proof}
	Under Assumption \ref{rho}, BSDEs \eqref{bsde3} admit only trivial solutions $(0,0)$. Since \(\theta_i = 1\) for all \(i=1,2,\cdots,n\), $\sum\limits_{k=1}^{n}Z^*_k=0$ and $\sum\limits_{k=1}^{n}z_k=0$. It follows from \eqref{phiz} that
	\begin{equation*}
		\Phi = \frac{\hat{\gamma}}{p(0)\check{h}(0)^2} \left( \check{\eta} + \left( \frac{\Lambda}{p}+\rho \right)\check{h} \right).
	\end{equation*}
Notice $\frac{\hat{\gamma}}{p(0)\check{h}(0)^2}\neq0$. Define
	\begin{equation*}
		\Xi \triangleq \check{\eta} + \left( \frac{\Lambda}{p}+\rho \right)\check{h}.
	\end{equation*}
We claim that there is no Nash equilibrium if $\Xi \not\equiv 0$, and there are infinitely many Nash equilibria if $\Xi\equiv0$.

Then we present an equivalent condition to $\Xi\equiv0$. We claim that
$\Xi\equiv0$ if and only if \eqref{equivalent2} holds.
	
\(\Xi \equiv 0 \implies\) \eqref{equivalent2}: Set $L(t) = p(t)\check{h}(t)$ for $t\in[0,T]$, and by It\^{o}'s formula, it follows that
	\begin{equation} \label{pht}
		\begin{aligned}
			\mathrm{d}L &= \left[ \check{h}\left(-r p+\rho ^2 p+2 \rho \Lambda+\frac{\Lambda^2}{p}\right)+\check{\eta}\left( \rho p+\Lambda \right) \right] \dt + \left( \check{\eta} p+\check{h} \Lambda\right) \dw, \\
			& = \left( -r L+\Xi( \rho p+\Lambda) \right) \dt + \left(-\rho L + p\Xi \right) \dw. \\
		\end{aligned}
	\end{equation}
If $\Xi \equiv 0$, we get from \eqref{pht} an SDE
	\begin{equation} \label{lsde2}
		\left\{\begin{array}{l}
			\mathrm{d}L = -rL\dt - \rho L\dw ,
			\quad t \in[0, T], \\
			L(T)= -1, \\
			L(0)= p(0)\check{h}(0).
		\end{array}\right.
	\end{equation}
The first two equations in \eqref{lsde2} give a unique explicit $\mathbb{F}$-adapted solution as below
	\begin{equation*}
		L(t)=-\exp \left( \int_t^T \rho(s) \dw_s + \int_t^T\left( r(s)+\frac{1}{2} \rho(s)^2 \right) \ds \right), \quad t \in [0, T].
	\end{equation*}
Hence, if SDE \eqref{lsde2} is solvable, then
$$\int_0^T \rho(s) \dw_s + \int_0^T\left( r(s)+\frac{1}{2} \rho(s)^2 \right) \ds=L(0)= p(0)\check{h}(0) ,$$ must be a constant, which implies that \eqref{equivalent2} is true.
	
\eqref{equivalent2} \(\implies \Xi \equiv 0\): If \eqref{equivalent2} holds, $\int_0^T\left( r(s)+\frac{1}{2}\rho(s)^2 \right) \ds+\int_0^T \rho(s) \dw$ is a constant, which together with $\int_0^t \rho(s) \dw_s + \int_0^t\left( r(s)+\frac{1}{2} \rho(s)^2 \right) \ds $ is $\mathcal{F}_t$-measurable, leads to $\int_t^T \rho(s) \dw_s + \int_t^T\left( r(s)+\frac{1}{2} \rho(s)^2 \right) \ds $ is $\mathcal{F}_t$-measurable for $t\in[0,T]$. According to \eqref{sxz15} and \eqref{hit}, we know
	\begin{equation*}
		p(t)\check{h}(t) = -\frac{\mathbb{E} \bigg[ \exp\Big\{\int_{t}^{T} -\rho \dw+ \int_{t}^{T}(-r- \frac{\rho^2}{2} ) \ds \Big\} \;\bigg{|}\; \mathcal{F}_t \bigg]}{\mathbb{E}\bigg[ \exp\Big\{\int_{t}^{T} -2 \rho \dw+ \int_{t}^{T}(-2r-\rho_i^2) \ds \Big\} \;\bigg{|}\; \mathcal{F}_t \bigg]}, 
	\end{equation*}
and due to measurability we further have
	\begin{equation*}
		p(t)\check{h}(t) = -\exp \left( \int_t^T \rho \dw + \int_t^T(r+\frac{1}{2} \rho^2) \ds \right).
	\end{equation*}
This shows that $L(t)=p(t)\check{h}(t)$ is a solution of SDE \eqref{lsde2}. Bearing in mind that $p(t)\check{h}(t)$ is also a solution of SDE \eqref{pht}, we immediately have $\Xi\equiv0$.
\end{proof}

Let us see the delicate balance between \(\chi\) and \(\Phi = 0\) under Assumption \ref{deter}.

\begin{thm} \label{criteria}
	Assume Assumption \ref{deter} and $\Psi = 1$. Then the existence of Nash equilibrium can be classified into the following situations.
\begin{enumerate}
		\item Infinitely Many Equilibria: If the equality
		\begin{equation} \label{equivalent1}
			\int_{0}^{T} \mathbb{E} \big[ q (\bar{\rho}\hat{\eta}' + G)\big] \ds = \int_{0}^{T} q (\bar{\rho}\hat{\eta}' + G) \ds + \int_{0}^{T} q \hat{\eta}' \dw ,
		\end{equation}
		holds, where
		\begin{equation}\label{sxz16}
			q(t) \triangleq e^{-\int_{0}^{t}r(s) \ds},\quad \hat{\eta}' \triangleq \sum_{i=1}^{n}\rho_i\frac{Y^*_i}{p_i},\quad G \triangleq \sum_{i=1}^n \sum_{j \neq i} \frac{\rho_j - \rho_i}{n}\rho_j \frac{Y^*_j}{p_j},
		\end{equation}
		there exist infinitely many Nash equilibria whose components are parameterized by \(\chi \in L_{\mathbb{F}}^2(0, T; \mathbb{R})\) as below
		\begin{equation} \label{strategy}
			\pi^*_i= \frac{\chi}{\sigma_i} - \frac{n-1}{n\sigma_i}\left( \tilde{\eta}_i' + \rho_i \frac{Y^*_i}{p_i} \right),
		\end{equation}
		where $\tilde{\eta}'_i$ is the solution to BSDE \eqref{lbsde2}.
		\item No Equilibrium: If \eqref{equivalent1} does not hold, no Nash Equilibrium exists.
	\end{enumerate}
\end{thm}
\begin{proof}

	Under Assumption \ref{deter}, \(c_i = 0\) for \(i=1,2,\cdots,n\) and thus \(C' = 0\). Then BSDE \eqref{lbsde2} admits a unique solution $(\tilde{\boldsymbol{h}}',\tilde{\boldsymbol{\eta}}')$ parameterized by \(\chi\). In this case \(\phi_i\) has a form
	\begin{equation}\label{sxz17}
		\phi_i = \tilde{\eta}_i' + \rho_i \frac{Y^*_i}{p_i}.
	\end{equation}
Denote by \(\vec{1} = (1, 1, \ldots, 1)^\top \in \mathbb{R}^n\) the unit vector. Noticing $\sum\limits_{i=1}^{n} A'_{ij}= -r$ and $\sum\limits_{i=1}^{n} B'_{ij}= -\bar{\rho}$, and using BSDE \eqref{lbsde2}, we get a scalar-valued BSDE:
	\begin{equation} \label{lbsde4}
		\left\{\begin{array}{l}
			\mathrm{d} \hat{h} =-\left\{-r\hat{h} - \bar{\rho} \hat{\eta} + G \right\} \dt+\hat{\eta} \dw, \quad t \in[0, T], \\
			\hat{h} (T)= 0,
		\end{array}\right.
	\end{equation}
	where $ \hat{h}\triangleq\vec{1}\cdot\tilde{\boldsymbol{h}}'$, $\hat{\eta}\triangleq \vec{1}\cdot\tilde{\boldsymbol{\eta}}'$ and \(G\triangleq\vec{1}\cdot F' \). 
Note that $G$ is independent of \(\chi\) and so is BSDE \eqref{lbsde4}. By \eqref{sxz16} and \eqref{sxz17} we know $ \hat{\eta} = \hat{\eta}'-\Phi$, so BSDE \eqref{lbsde4} involves $\Phi$. 
Clearly, if $\Phi=0$, there are infinitely many equilibria given by \eqref{strategy}, and otherwise there is no Nash Equilibrium. So we need to prove that \eqref{equivalent1} is a criteria to determine $\Phi=0$ or not.

	If $\Phi = 0$, BSDE \eqref{lbsde4} can be rewritten as the following SDE 
\begin{equation} \label{sde1}
		\left\{\begin{aligned}
			& \mathrm{d}\hat{h}= -\left\{-r\hat{h} +\bar{\rho}\hat{\eta}' + G \right\} \dt - \hat{\eta}' \dw, \quad t \in[0, T], \\
			& \hat{h}(T)=0.
		\end{aligned}\right.
	\end{equation}
By SDE \eqref{sde1} and It\^o's lemma, we obtain 
	\begin{equation*}
		\hat{h}(t) = \frac{1}{q(t)}\bigg[ \hat{h}(0) - \int_{0}^{t} q (\bar{\rho}\hat{\eta}' + G) \ds - \int_{0}^{t} q \hat{\eta}' \dw \bigg].
	\end{equation*}
Letting $t = T$ and taking expectation, we have
	\begin{equation*}
		\hat{h}(0) = \int_{0}^{T} \mathbb{E} \big[ q (\bar{\rho}\hat{\eta}' + G)\big] \ds.
	\end{equation*}
Hence
	\begin{equation*}
		\hat{h}(t) = \frac{1}{q(t)}\bigg[ \int_{0}^{T} \mathbb{E} \big[ q (\bar{\rho}\hat{\eta}' + G)\big] \ds - \int_{0}^{t} q (\bar{\rho}\hat{\eta}' + G) \ds - \int_{0}^{t} q \hat{\eta}' \dw \bigg].
	\end{equation*}
If the equality \eqref{equivalent1} holds, the constraint $\hat{h}(T)=0$ is satisfied, and the constrained SDE \eqref{sde1} is well-posed. 

On the other hand, if the constrained SDE \eqref{sde1} is well-posed with a solution $\hat{h}$, then $(\hat{h},-\hat{\eta}')$ is also a solution to BSDE \eqref{lbsde4}. By the uniqueness of solution to BSDE \eqref{lbsde4}, we know $-\hat{\eta}'= \hat{\eta}$ which implies $\Phi = \hat{\eta} + \hat{\eta}' = 0$.
\end{proof}

\section{Example}\label{sec:example}
In the case that both Assumptions \ref{rho} and \ref{deter} are satisfied, explicit expressions for the feedback strategies can be given (if existing). We will show this in the following theorem, which can also be regarded as a special example for our theoretical results.
\begin{thm}
	Assume Assumptions \ref{rho} and \ref{deter} hold. Then there exists a unique Nash equilibrium if $\Psi < 1$ and no Nash equilibrium exists if $\Psi = 1$.
\end{thm}

\begin{proof}
	Under Assumptions \ref{rho} and \ref{deter}, $r$ and $\rho$ are deterministic functions, and
	\begin{equation*}
		p(t) = \exp((-2r+\rho^2)(t-T)), \quad \check{h}(t)=-\exp\{r(t-T)\}, \quad \Lambda = \tilde{h}=\tilde{\eta} = \check{\eta}=0,
	\end{equation*}
where $(p,\Lambda)$, $(\tilde{h},\tilde{\eta})$ and $(\check{h},\check{\eta})$ are the solutions to BSDE \eqref{bsde1}, \eqref{bsde3} and \eqref{bsde4}, respectively.
	
If $\Psi < 1$, by Remark \ref{sxz19} and Theorems \ref{sxz18} and \ref{deternashthm}, there exists a unique Nash equilibrium. In this case the optimal feedback strategy \eqref{control2} reduces to
	\begin{equation*} 
		\pi_i^* = \theta_i \frac{\widehat{(\sigma\pi^*)}_i }{\sigma_i} - \frac{\rho}{\sigma_i}\left( Z^*_i - e^{rt} z_i - \frac{ \exp \left(r(t-T)+\rho^2 T \right)}{\gamma_i} \right).
	\end{equation*}
And $\phi_i$ has a form
	\begin{equation*}
		\phi_i = \rho \left( Z^*_i - e^{rt} z_i - \frac{ \exp \left(r(t-T)+\rho^2 T \right)}{\gamma_i} \right).
	\end{equation*}
Hence the unique Nash equilibrium is
	\begin{equation*}
		\sigma_i\pi^*_i = -\frac{1}{1-\Psi} \frac{n\theta_i}{n-1+\theta_i}\sum_{i=1}^{n}\frac{\phi_i}{n+\frac{n\theta_i}{n-1}} - \frac{\phi_i}{1+\frac{\theta_i}{n-1}}.
	\end{equation*}
	
If $\Psi = 1$, no equilibrium exists. It follows from 
the term in Theorem \ref{criterion2} that 
	\begin{equation*}
		\Xi = \check{\eta} + \left( \frac{\Lambda}{p}+\rho\right)\check{h} = \rho\check{h} \not\equiv 0,
	\end{equation*}
	or equivalently, \eqref{equivalent2} does not hold since $r+\frac{1}{2}\rho^2$ is deterministic and $\int_0^T \rho \dw\not\equiv0$. Alternatively, we obtain the same conclusion by checking \eqref{equivalent1} in Theorem \ref{criteria}. If \eqref{equivalent1} holds,
	\begin{equation*}
		\int_{0}^{T} \mathbb{E} \left[ q \bar{\rho}\hat{\eta}'\right] \ds = \int_{0}^T q \bar{\rho}\hat{\eta}' \, \ds + \int_{0}^{T} q \hat{\eta}' \, \dw.
	\end{equation*}
Since
	\begin{equation*}
		\bar{\rho}\hat{\eta}' = \sum_{i=1}^{n} \rho^2\frac{Y^*_i}{p} = \sum_{i=1}^{n} \rho^2\left(Z^*_i + \left(-\frac{z_i}{\check{h}(0)}+\frac{1}{\gamma_ip(0)\check{h}(0)^2}\right)\check{h} \right) = \sum_{i=1}^{n} \frac{\hat{\gamma} \rho^2}{p(0)\check{h}(0)^2}\check{h},
	\end{equation*}
	we see
$q \bar{\rho}\hat{\eta}'$ is deterministic and $\hat{\eta}'\not\equiv0$. But $\int_{0}^{T} q \hat{\eta}' \, \dw$ obeys the law of normal distribution, which results in a contradiction.
\end{proof}

\section{Conclusion}\label{sec:conclusion}

In this paper, we investigate time-inconsistent Nash equilibrium strategies for a multi-agent game under MV criterion. We first solve a linearly constrained stochastic LQ control problem to derive optimal strategies for each agent. Then we use a decoupling technique to establish a connection between the Nash equilibrium and a novel type of linear multi-dimensional BSDEs \eqref{lbsde1}. The well-posedness of such BSDEs is studied in both the usual case and the marginal case. Based on Assumptions \ref{rho} and \ref{deter}, we have more refined analyses of Nash equilibria, as summarized in the following table.
\begin{table}[H]
	\begin{spacing}{1.0}
		\centering
		\caption{Existence of Nash Equilibria} \label{conclusion}
		\begin{tabular}{l|c|c|c|c}
			\hline
			{\bf Assumptions} & {\bf None} & {\bf \ref{rho}} & {\bf \ref{deter}} & {\bf \ref{rho} and \ref{deter}} \\
			\hline
			$\Psi < 1$ & Discussion & Unique& Unique & Unique \\
			\hline
			$\Psi = 1$ & Open & None or infinity & None or infinity & None \\
			\hline
		\end{tabular}
	\end{spacing}
\end{table}

\appendix
\section{Well-Posedness of a New Type of Nonlinear BSDEs}\label{appendix}
\begin{lemma}
Consider the following BSDE
	\begin{equation} \label{newbsde}
		Y(t)=\xi+\int_t^T f\big(s, Y(s), Z(s), C(s)Y(0)\big) \ds-\int_t^T Z(s) \dw(s), \quad t \in[0, T],
	\end{equation}
	where
	\begin{enumerate}
		\item[{\rm (a1)}] $\left\{W(t)\right\}_{t \in[0, T]}$ is a standard $m$-dimensional Brownian motion with its natural filtration denoted by $\mathbb{F}=\left\{\mathcal{F}_t\right\}_{t \in[0, T]}$;
		\item[{\rm (a2)}] $\xi \in \mathcal{F}_T$ and $\mathbb{E}|\xi|^2<\infty$; 
		\item[{\rm (a3)}] $f: \Omega \times[0, T] \times \mathbb{R}^n \times \mathbb{R}^{n \times m} \times \mathbb{R}^n \rightarrow \mathbb{R}^n$ is $\mathscr{P} \otimes \mathcal{B}(\mathbb{R}^n) \otimes \mathcal{B}(\mathbb{R}^{n \times m})\otimes \mathcal{B}(\mathbb{R}^n) $-measurable and $C: \Omega \times[0, T] \rightarrow \mathbb{R}^{n\times n}$ is $\mathscr{P}$-measurable, where $\mathscr{P}$ is the predictable sub-$\sigma$ algebra of $\mathcal{F} \otimes \mathcal{B}([0,T])$;
		\item[{\rm (a4)}] for any $t \in[0, T], y_1, y_2 \in \mathbb{R}^n, z_1, z_2 \in \mathbb{R}^{n \times m}$, there exists a Lipschitz constant $L \geq 0$ such that
		\begin{equation*}
			|f\left(t, y_1, z_1, \zeta_1 \right)-f\left(t, y_2, z_2,\zeta _2 \right)|\leq L\left(|y_1-y_2|+|z_1-z_2| + |\zeta_1 - \zeta_2 |\right);
		\end{equation*}
		\item[{\rm (a5)}] $ \mathbb{E}\int_0^T(|f(t, 0,0,0)|^2+|C(t)|^2) \dt<\infty$.
	\end{enumerate}
	Fix the terminal value $\xi$, the driver $f$ and process $C$, then
BSDE \eqref{newbsde} admits a unique solution $( Y, Z) \in S_{\mathbb{F}}^2\left(0, T ; \mathbb{R}^n\right) \times L_{\mathbb{F}}^2\left(0, T ; \mathbb{R}^{n\times m}\right)$, provided that the horizon time $T>0$ is sufficiently small.
\end{lemma}

\begin{proof}
For a given $v\in\mathbb{R}^n$, consider BSDE
	\begin{equation*}
		Y^v(t)=\xi+\int_t^T f\big(s, Y^v(s), Z^v(s), C(s)v\big) \ds-\int_t^T Z^v(s) \dw(s), \quad t \in [0, T].
	\end{equation*}
It is well known that, under (a1)--(a5), the above BSDE admits a unique solution $( Y^v, Z^v) \in S_{\mathbb{F}}^2\left(0, T ; \mathbb{R}^n\right) \times L_{\mathbb{F}}^2\left(0, T ; \mathbb{R}^{n\times m}\right)$.

Define the mapping \( \Phi: \mathbb{R}^n \to \mathbb{R}^n \) by \( \Phi(v) = Y^v(0) \). Below, we analyze the contraction property of \( \Phi \).

Consider the following two BSDEs with parameters $\hat{v}$ and $\tilde{v}\in\mathbb{R}^n$:
	\begin{equation*}
		\begin{cases}
			\displaystyle \hat{Y}^{\hat{v}}(t) = \xi + \int_t^T \big( f(s, \hat{Y}^{\hat{v}}(s), \hat{Z}^{\hat{v}}(s),C(s)\hat{v}) \big) \ds - \int_t^T \hat{Z}^{\hat{v}}(s) \dw(s), \\
			\displaystyle \tilde{Y}^{\tilde{v}}(t) = \xi + \int_t^T \big( f(s, \tilde{Y}^{\tilde{v}}(s), \tilde{Z}^{\tilde{v}}(s)), C(s)\tilde{v} \big) \ds - \int_t^T \tilde{Z}^{\tilde{v}}(s) \dw(s).
		\end{cases}
	\end{equation*}
	For $\beta = 16(L^2+1)$, the standard estimate for BSDEs yields
	\begin{equation*} 
\begin{aligned}
& |\hat{Y}^{\hat{v}}(t) -\tilde{Y}^{\tilde{v}}(t)|^2 + \frac{1}{2} \mathbb{E}^{\mathcal{F}_t} \left[ \int_t^T e^{\beta(s-t)} \left( |\hat{Y}^{\hat{v}}(s) -\tilde{Y}^{\tilde{v}}(s)|^2 + |\hat{Z}^{\hat{v}}(s) -\tilde{Z}^{\tilde{v}}(s)|^2 \right) \ds \right]\\
\leq& \frac{L^2}{4(L^2 + 1)} \mathbb{E}^{\mathcal{F}_t} \int_t^T e^{\beta(s-t)} |C(s)|^2 |\hat{v}-\tilde{v}|^2 \ds.
\end{aligned}	
\end{equation*}
In particular, it implies 	\begin{equation*}
		|\hat{Y}^{\hat{v}}(0) -\tilde{Y}^{\tilde{v}}(0)|^2 \leq \frac{1}{4} \int_0^T e^{\beta s}\mathbb{E}\big[ |C(s)|^2\big] \ds \cdot |\hat{v}-\tilde{v}|^2.
	\end{equation*}
Taking \( T> 0 \) sufficiently small such that \( \int_0^T e^{\beta s}\mathbb{E}\big[ |C(s)|^2\big] \ds <1,\) the above estimate leads to i.e.,
	\begin{equation*}
		|\Phi(v_1) - \Phi(v_2)|\leq \frac{1}{2} |v_1 - v_2|,
	\end{equation*}
	so that $\Phi$ is a contraction mapping.
By the fixed-point theory, there exists a unique \( v \in \mathbb{R}^n \) such that \( \Phi(v) = v \), i.e., $Y^v(0)=v$, from which the existence and uniqueness of solution to BSDE \eqref{newbsde} follows.
\end{proof}

\bibliographystyle{siam}

\end{document}